\documentclass[dvipsnames]{siamart190516}

\usepackage{lineno,hyperref}
\modulolinenumbers[5]
\usepackage{amsmath}
\usepackage{amsfonts,amssymb}
\usepackage{graphicx}
\usepackage{bm}
\graphicspath{{./Figures/}}
\usepackage{color}
\usepackage{ulem}
\usepackage{multirow}
\usepackage{multicol}
\usepackage{algorithm}
\usepackage{algpseudocode}
\usepackage{algorithmicx}
\usepackage{hyperref}
\usepackage{upquote}
\usepackage{cprotect}
\usepackage{setspace}

\usepackage[showonlyrefs]{mathtools}
\usepackage{xcolor}

\usepackage[defaultcolor=red]{changes}

\DeclareMathAlphabet{\mathbf}{OT1}{cmr}{bx}{it}

\newcommand{\vb}{{\mathbf b}}

\newcommand{\ve}{{\mathbf e}}

\newcommand{\vf}{{\mathbf f}}

\newcommand{\vv}{{\mathbf v}}

\newcommand{\vw}{{\mathbf w}}
\newcommand{\vx}{{\mathbf x}}
\newcommand{\vy}{{\mathbf y}}
\newcommand{\vz}{{\mathbf z}}

\newcommand{\vnull}{\boldsymbol{0}}

\newcommand{\spK}{{\cal K}}

\newcommand{\lmin}{\lambda_{\min}}
\newcommand{\lmax}{\lambda_{\max}}

\newcommand{\dmu}{\d\mu(t)}

\renewcommand{\d}{\,\mathrm{d}}

\newcommand{\R}{\mathbb{R}}

\newcommand{\Rm}{\mathbb{R}^m}

\newcommand{\Rmm}{\mathbb{R}^{m \times m}}
\newcommand{\C}{\mathbb{C}}
\newcommand{\Cn}{\mathbb{C}^n}
\newcommand{\Cm}{\mathbb{C}^m}

\newcommand{\Cnn}{\mathbb{C}^{n \times n}}

\newcommand{\vfmopt}{\vf_m^{\textnormal{opt}}}

\DeclareMathOperator{\Span}{span}

\DeclareMathOperator{\spec}{spec}

\DeclareMathOperator{\diag}{diag}

\newtheorem{remarksimple}[theorem]{Remark}
\let\oldremarksimple\remarksimple
\renewcommand{\remarksimple}{\oldremarksimple\normalfont}
\newenvironment{remark}{\begin{remarksimple}}{\hfill$\diamond$\end{remarksimple}}
\newtheorem{experimentsimple}[theorem]{Experiment}
\let\oldexperimentsimple\experimentsimple
\renewcommand{\experimentsimple}{\oldexperimentsimple\normalfont}

\newtheorem{examplesimple}[theorem]{Example}
\let\oldexamplesimple\examplesimple
\renewcommand{\examplesimple}{\oldexamplesimple\normalfont}
\newenvironment{example}{\begin{examplesimple}}{\hfill$\diamond$\end{examplesimple}}

\usepackage{pgfplots}
\usepgfplotslibrary{groupplots}
\usetikzlibrary{external}

\pgfkeys{/pgf/images/include external/.code={\includegraphics{#1}}}

\usetikzlibrary{decorations.markings}
\makeatletter
\tikzset{
  nomorepostactions/.code={\let\tikz@postactions=\pgfutil@empty},
  mymark/.style 2 args={decoration={markings,
    mark= between positions 0 and 1 step (1/9)*\pgfdecoratedpathlength with{%
        \tikzset{#2,every mark}\tikz@options
        \pgfuseplotmark{#1}%
      },  
    },
    postaction={decorate},
    /pgfplots/legend image post style={
        mark=#1,mark options={#2},every path/.append style={nomorepostactions}
    },
  },
}
\makeatother

\pgfplotscreateplotcyclelist{list_bounds}{%
Cerulean,line width=1.25pt, mark=*, solid, mark repeat=10\\
BurntOrange,line width=1.25pt, mark=diamond, dashed, mark repeat=10, mark options={solid}\\
OliveGreen,line width=1.25pt, mark=x, dashdotted, mark repeat=10, mark options={solid}\\
Orchid,line width=1.25pt, mark=square, dashdotted, mark repeat=10, mark options={solid}\\
BrickRed,line width=1.25pt, mark=o, dashdotted, mark repeat=10, mark options={solid}\\
LimeGreen,line width=1.25pt, mark=triangle, dashdotted, mark repeat=10, mark options={solid}\\
}
\pgfplotsset{compat=1.17}

\Crefname{remarksimple}{Remark}{Remarks}
\title{Near instance optimality of the Lanczos method for Stieltjes and related matrix functions}
\author{Marcel Schweitzer\thanks{School of Mathematics and Natural Sciences, Bergische Universit\"at Wuppertal, 42097 Wuppertal, Germany, \email{marcel@uni-wuppertal.de}.}}

\begin{document}
\maketitle

\pagestyle{myheadings} \thispagestyle{plain}
\markboth{M. SCHWEITZER}{NEAR OPTIMALITY OF LANCZOS FOR MATRIX FUNCTIONS}

\begin{abstract}
Polynomial Krylov subspace methods are among the most widely used methods for approximating $f(A)\vb$, the action of a matrix function on a vector, in particular when $A$ is large and sparse. When $A$ is Hermitian positive definite, the Lanczos method is the standard choice of Krylov method, and despite being very simplistic in nature, it often outperforms other, more sophisticated methods. In fact, one often observes that the error of the Lanczos method behaves almost exactly as the error of the best possible approximation from the Krylov space (which is in general not efficiently computable). However, theoretical guarantees for the deviation of the Lanczos error from the optimal error are mostly lacking so far (except for linear systems and a few other special cases). We prove a rigorous bound for this deviation when $f$ belongs to the important class of Stieltjes functions (which, e.g., includes inverse fractional powers as special cases) and a related class (which contains, e.g., the square root and the shifted logarithm), thus providing a \emph{near instance optimality} guarantee. While the constants in our bounds are likely not optimal, they greatly improve over the few results that are available in the literature and resemble the actual behavior much better.
\end{abstract}



\section{Introduction}\label{sec:intro}
Approximating the action of a matrix function $f(A)\vb$, where $A\in\Cnn$ is a Hermitian matrix, $f$ is defined on the spectrum of $A$ and $\vb \in \Cn$ is a vector, plays an important role in many areas of applied mathematics, scientific computing and data science, including the solution of (fractional) differential equations~\cite{GarrappaPopolizio2018,HochbruckOstermann2010}, the analysis of complex networks~\cite{BenziBoito2020,EstradaRodriguezVelaszquez2005}, Gaussian process regression~\cite{pleiss2020fast,Stoll2020} and theoretical particle physics~\cite{VanDenEshofFrommerLippertSchillingVanDerVorst2002,Neuberger1998}, among many others. 

In these applications, $A$ is typically huge and sparse (or structured in some other way), such that matrix-vector products with it can efficiently be computed, while most other operations (like computing matrix factorizations) are infeasible due to high computational cost and storage demands. In this setting, it is infeasible to first compute the matrix function $f(A)$ and then multiply it to $\vb$. Instead, one aims to directly approximate the solution vector $f(A)\vb$ by means of an iterative method. The by far most popular choice for this task are (polynomial) Krylov subspace methods~\cite{DruskinKnizhnerman1989,Saad1992} based on the Arnoldi process~\cite{Arnoldi1951}. When $A$ is Hermitian, the Arnoldi process simplifies to the short-recurrence Lanczos method~\cite{Lanczos1950}, which is the focus of this work. 

The Lanczos algorithm is given in~\cref{alg:lanczos}. Note that for ease of notation, we assume throughout the paper---without loss of generality---that $\|\vb\| = 1$, where $\|\cdot\|$ denotes the Euclidean norm. The Lanczos method constructs an orthonormal basis $\vv_1,\dots,\vv_m$ of the Krylov subspace
\[
\spK_m(A,\vb) := \Span\{\vb,A\vb,\dots,A^{m-1}\vb\} = \{p_{m-1}(A)\vb : p_{m-1} \in \Pi_{m-1}\},
\]
where $\Pi_{m-1}$ denotes the space of all polynomials of degree at most $m-1$, by exploiting that the basis vectors fulfill a three term recurrence $\beta_{j+1}\vv_{j+1} = A\vv_j - \alpha_j\vv_j - \beta_{j-1}\vv_{j-1}$. Collecting the recurrence coefficients in a tridiagonal matrix
\[
T_m = \begin{bmatrix}
    \alpha_1 & \beta_2  &             &              &          \\
    \beta_2  & \alpha_2 & \beta_3     &              &          \\
             & \ddots   & \ddots      & \ddots       &          \\
             &          & \beta_{m-1} & \alpha_{m-1} & \beta_m  \\
             &          &             & \beta_m      & \alpha_m
\end{bmatrix} \in \Rmm
\]
and the basis vectors in $V_m = [\vv_1 \mid \dots | \vv_m] \in \C^{n \times m}$, we have the Lanczos relation
\begin{equation}\label{eq:lanczos_relation}
AV_m = V_mT_m + \beta_{m+1} \vv_{m+1}(\ve_m^{(m)})^\ast,
\end{equation}
where $\ve_i^{(m)} \in \Rm$ denotes the $i$th canonical unit vector in $\Rm$ and $(\ \cdot\ )^\ast$ denotes the conjugate transpose of a vector (or a matrix). An immediate consequence of~\eqref{eq:lanczos_relation} is that 
\[
T_m = V_m^\ast A V_m.
\]
Given the quantities in~\eqref{eq:lanczos_relation}, the $m$th \emph{Lanczos approximation} for $f(A)\vb$ is given by
\begin{equation}\label{eq:lanczos_approximation}
\vf_m := V_m f(V_m^\ast A V_m)V_m^\ast \vb = V_m f(T_m) \ve_1^{(m)}.
\end{equation}
A remarkable property of the Lanczos approximation is that (in exact arithmetic) it is guaranteed to yield the exact vector $f(A)\vb$ after a finite number of iterations: Denoting by $M$ the \emph{invariance index} of the Krylov subspace (i.e., the smallest $M$ for which $\spK_{m+1}(A,\vb) = \spK_m(A,\vb)$ for all $m \geq M$), it is well known that $\vf_M = f(A)\vb$~\cite[Theorem~3.6]{Saad1992}. Clearly, as $\spK_m(A,\vb)$ is a subspace of $\Cn$, we have $M \leq n$, so that $f(A)\vb$ is found after at most $n$ iterations.

\begin{algorithm}[t]
\begin{algorithmic}[1]
\State $\vv_0 \leftarrow \vnull^{(n)}$
\State $\vv_1 \leftarrow \vb$
\State $\beta_1 \leftarrow 0$
\For{$j = 1,\dots,m$}
	\State $\vw_j \leftarrow  A\vv_j - \beta_j\vv_{j-1}$
	\State $\alpha_j \leftarrow \vv_j^\ast\vw_j$
	\State $\vw_j \leftarrow \vw_j - \alpha_j\vv_j$
	\State $\beta_{j+1} \leftarrow \|\vw_j\|$ \label{line:beta}
	\State $\vv_{j+1} \leftarrow (1/\beta_{j+1}) \vw_j$
\EndFor
\end{algorithmic}
\caption{Lanczos method for constructing an ONB of $\spK_m(A,\vb)$}
    \label{alg:lanczos}
\end{algorithm}

The invariance of $\spK_M(A,\vb)$ is associated with $\beta_{M+1}$ being zero in line~8 of~\Cref{alg:lanczos}. As it indicates that the exact solution is found, this event is also referred to as a \emph{lucky breakdown}.

A lot less is known about the approximation quality of $\vf_m$ for $m < M$. A famous result, sometimes called the \emph{near-optimality} or \emph{quasi-optimality} property of Lanczos, states that
\begin{equation}\label{eq:near_fov_optimality}
\|f(A)\vb-\vf_m\| \leq 2 \min_{p \in \Pi_{m-1}}\ \max_{z \in [\lmin,\lmax]} |f(z) - p(z)|,
\end{equation}
where $\lmin$ and $\lmax$ denote the smallest and largest eigenvalue of $A$, respectively; see, e.g.,~\cite{BeckermannReichel2009,Saad1992}. This bound relates the error of $\vf_m$ to best polynomial approximation on $[\lmin,\lmax]$, the field of values (FOV) of $A$. To distinguish it from other types of near-optimality, the authors of~\cite{AmselChenGreenbaumMuscoMusco2023} propose the more precise term \emph{``near FOV optimality''} for~\eqref{eq:near_fov_optimality}. While~\eqref{eq:near_fov_optimality} can be used to derive a priori bounds on the Lanczos error (see, e.g.,~\cite{BeckermannReichel2009}) by exploiting results from polynomial approximation theory, these bounds need not be descriptive of the actual behavior of Lanczos for a specific problem instance $(f,A,\vb)$. Clearly, the right hand side of \eqref{eq:near_fov_optimality} is the same for any $A$ with spectral interval $[\lmin,\lmax]$ (irrespective of the distribution of the eigenvalues inside this interval) and for any $\vb$ (irrespective of the contribution of the individual eigenvectors of $A$ to $\vb$), and in that sense it gives a ``worst case'' bound, as it needs to be valid for \emph{any} $A$ with field of values $[\lmin,\lmax]$ and \emph{any} vector $\vb$.

In this work, we are therefore interested in a stronger optimality concept, which is dubbed \emph{``near instance optimality''} in~\cite{AmselChenGreenbaumMuscoMusco2023}. We want to find $1 \leq C < \infty$ such that
\begin{equation}\label{eq:near_instance_optimality}
\|f(A)\vb-\vf_m\| \leq C\min_{\vx \in \spK_m(A,\vb)} \|f(A)\vb-\vx\| = C\min_{p \in \Pi_{m-1}} \|f(A)\vb-p(A)\vb\|.
\end{equation}
Near instance optimality (or the slightly weaker concept of near spectrum optimality; see \Cref{sec:existing_results}) is very important for theoretically understanding the behavior of the Lanczos method, as it can, e.g., form the basis of superlinear convergence results~\cite{BeckermannKuijlaars2001,BeckermannKuijlaars2002}.

Our main result, \Cref{thm:near_instance_optimality_stieltjes}, proves a near instance optimality guarantee of the form~\eqref{eq:near_instance_optimality}---and gives an explicit expression for $C$---for the case that $f$ is a Stieltjes (or Markov) function, i.e.,
\begin{equation}\label{eq:stieltjes_function}
f(z) = \int_0^\infty \frac{1}{z+t} \dmu,
\end{equation}
where $\mu: [0, \infty) \longrightarrow \R$ is monotonically increasing and such that $\int_0^\infty \frac{1}{1+t} \dmu < \infty$. This class of functions, e.g., contains inverse fractional powers as important special cases~\cite{Berg2007} and is frequently studied in numerical analysis, as the integral representation~\eqref{eq:stieltjes_function} allows to transfer results for shifted inverses to general matrix functions, which can be beneficial both from a theoretical and an algorithmic point of view; see, e.g.,~\cite{BenziSimoncini2017,FrommerGuettelSchweitzer2014a,FrommerGuettelSchweitzer2014b,FrommerSchweitzer2015,GuettelKnizhnerman2013,GuettelSchweitzer2021,MasseiRobol2020,Schweitzer2016thesis}. Important properties of Stieltjes functions as well as further examples of functions belonging to this class are given in~\Cref{appendix:stieltjes_functions}.

The remainder of this paper is organized as follows. In \Cref{sec:existing_results}, we review the few known near instance optimality results that are available in the literature so far. In \Cref{sec:near_optimality_stieltjes}, we present our main near instance optimality result together with several technical lemmas required for its proof. \Cref{sec:related_classes} discusses the extension of our main result to related function classes, in particular to functions of the form $f(z) = zg(z)$, where $g$ is a Stieltjes function. We illustrate our results by some examples in \Cref{sec:examples} and compare them to results from the literature. Concluding remarks are given in \Cref{sec:conclusion}.

Throughout the paper, we assume exact arithmetic.

\section{Existing near instance optimality results for $f(A)\vb$}\label{sec:existing_results}

Near instance optimality guarantees for the Lanczos approximation only exist for a quite limited number of special cases. Certainly the most famous such result is concerned with the special case $f(z) = z^{-1}$, which means that $f(A)\vb$ corresponds to the solution of the linear system $A\vx = \vb$. In this case, when $A$ is Hermitian positive definite, $\vf_m$ corresponds to the conjugate gradient approximation for $\vx$~\cite{HestenesStiefel1952}, which is known to be optimal in the $A$-norm $\|\vv\|_A = \sqrt{\vv^\ast A\vv}$, i.e.,
\begin{equation}\label{eq:instance_optimality_cg}
\|f(A)\vb-\vf_m\|_A = \min_{\vx \in \spK_m(A,\vb)} \|A^{-1}\vb - \vx\|_A.
\end{equation}
Thus, if one replaces the Euclidean norm by the $A$-norm in~\eqref{eq:near_instance_optimality}, the inequality holds with $C=1$ (and is therefore an equality). I.e., the conjugate gradient method is \emph{instance optimal} with respect to the $A$-norm. Of course, this also directly implies a near instance optimality guarantee in the Euclidean norm,
\begin{equation}\label{eq:near_instance_optimality_cg_euclidean}
\|f(A)\vb-\vf_m\| \leq \sqrt{\kappa(A)} \min_{\vx \in \spK_m(A,\vb)} \|A^{-1}\vb - \vx\|,
\end{equation}
where $\kappa(A) = \frac{\lmax}{\lmin}$ denotes the spectral condition number of $A$. For the case of non-Hermitian $A$, near optimality of the full orthogonalization method (FOM) is studied in~\cite{chen2024near}.

For $f$ different from the inverse, only very few results exist, and these give much weaker guarantees than~\eqref{eq:instance_optimality_cg}--\eqref{eq:near_instance_optimality_cg_euclidean}. Recent work in this direction has been done in~\cite{AmselChenGreenbaumMuscoMusco2023}, where are slightly looser concept of near instance optimality is used. In particular, in the minimization on the right hand side, $\min_{p \in \Pi_{m-1}} \|f(A)\vb-p(A)\vb\|$ is replaced by $\min_{p \in \Pi_{cm-1}} \|f(A)\vb-p(A)\vb\|$, for some $0 < c \leq 1$. If $c<1$, this means that the error of the Lanczos approximation is compared to the error of an optimal polynomial approximation of a \emph{lower} degree.

The first main result of~\cite{AmselChenGreenbaumMuscoMusco2023} is concerned with rational functions $f(z) = \frac{q(z)}{r(z)}$, where $q \in \Pi_{k}, r \in \Pi_\ell$. Denoting the zeros of $r$ by $z_i, i = 1,\dots,\ell$, and assuming that $m \geq \max\{k,\ell-1\}$,~\cite[Theorem~4]{AmselChenGreenbaumMuscoMusco2023} states that
\begin{equation}\label{eq:bound_amsel_etal}
\|f(A)\vb-\vf_m\| \leq \ell \cdot \left(\prod_{i=1}^\ell \kappa(A-z_iI)\right) \min_{\vx \in \spK_{m-\ell+1}(A,\vb)}\|f(A)\vb-\vx\|,
\end{equation}
i.e., the near optimality guarantee holds with $C= \ell \cdot \left(\prod_{i=1}^\ell \kappa(A-z_iI)\right)$ and $c = 1 - \frac{\ell-1}{m}$. If $A$ is Hermitian positive definite and all poles $z_i$ lie on the negative real axis, the bound can be simplified to
\[
\|f(A)\vb-\vf_m\| \leq \ell \cdot \kappa(A)^\ell \min_{\vx \in \spK_{m-\ell+1}(A,\vb)}\|f(A)\vb-\vx\|.
\]
In~\cite[Section~2.2]{AmselChenGreenbaumMuscoMusco2023}, implications for more general functions, which are well approximated by rational functions, are discussed. One shortcoming of~\eqref{eq:bound_amsel_etal} in this context is the exponential growth of the constant $C$ with respect to the degree $\ell$ of the denominator polynomial. In particular, the result can thus not straightforwardly be extended to general functions by a limiting argument, as $C \rightarrow \infty$ for growing $\ell$.

The second main result of~\cite{AmselChenGreenbaumMuscoMusco2023} concerns the square root $f(z) = \sqrt{z}$ and inverse square root $f(z) = \frac{1}{\sqrt{z}}$, two functions which are also covered by the analysis in the present paper; cf.~\Cref{sec:near_optimality_stieltjes,sec:related_classes}. It is not a near instance optimality guarantee, but a \emph{``near spectrum optimality''} guarantee, i.e., a bound similar to~\eqref{eq:near_fov_optimality}, where the interval $[\lmin,\lmax]$ on the right hand side is replaced by $\spec(A)$, the discrete set of eigenvalues of $A$. Specifically,~\cite[Theorem~6~\&~7]{AmselChenGreenbaumMuscoMusco2023} state that
\begin{equation}\label{eq:near_spectrum_optimality_invsqrt}
    \|A^{-1/2}\vb - \vf_m\| \leq \frac{3\kappa(A)}{\sqrt{\pi m}} \min_{p\in\Pi_{m/2-1}}\left(\max_{z \in \spec(A) } \left|\frac{1}{\sqrt{z}}-p(z)\right|\right)
\end{equation}
and
\begin{equation}\label{eq:near_spectrum_optimality_sqrt}
    \|A^{1/2}\vb - \vf_m\| \leq \frac{3\kappa(A)^2}{m^{3/2}} \min_{p\in\Pi_{m/2}}\left(\max_{z \in \spec(A) \cup \{0\}} \left|\sqrt{z}-p(z)\right|\right).
\end{equation}
While~\eqref{eq:near_spectrum_optimality_invsqrt}--\eqref{eq:near_spectrum_optimality_sqrt} involve a much smaller constant than the bound~\eqref{eq:near_instance_optimality} for rational functions, their main shortcoming is that $c = \frac{1}{2}$, i.e., the polynomial degree is halved. This typically means that the bound does not accurately reflect the actual convergence slope of the Lanczos approximation; cf.~\cite[Figure~7]{AmselChenGreenbaumMuscoMusco2023} as well as \Cref{sec:examples} below.

\begin{remark}\label{rem:spectrum_instance}
Concerning the concepts of near spectrum and near instance optimality, it is worth mentioning that the latter is the stronger concept, i.e., near instance optimality implies near spectrum optimality. Clearly, 
\[
\|f(A)\vb-p(A)\vb\| \leq \|f(A)-p(A)\| = \max_{\lambda \in \spec(A)} |f(\lambda)-p(\lambda)|
\]
so that~\eqref{eq:near_instance_optimality} implies
\[
\|f(A)\vb-\vf_m\| \leq C\min_{p \in \Pi_{m-1}} \|f(A)\vb-p(A)\vb\| \leq C\min_{p \in \Pi_{m-1}} \max_{\lambda \in \spec(A)} |f(\lambda)-p(\lambda)|.
\]
Interestingly, under certain assumptions on $\vb$, the converse is also true, i.e., near spectrum optimality implies near instance optimality. One situation in which this is the case is when $\vb$ has independent and identically distributed Gaussian entries; see~\cite[Appendix C.1]{AmselChenGreenbaumMuscoMusco2023}.
\end{remark}

Another near instance optimality result from the literature is concerned with the matrix exponential: In~\cite{DruskinGreenbaumKnizhnerman1998}, it is shown that
\begin{equation}\label{eq:near_instance_optimality_exponential}
\|\exp(-tA)\vb-\vf_m\| \leq 3 \|A\|^2 t^2 \max_{0 \leq s \leq t} \left(\min_{p \in \Pi_{m-3}} \|\exp(-sA)\vb - p(A)\vb\|\right).
\end{equation}
Note that~\eqref{eq:near_instance_optimality_exponential} is not exactly of the form~\eqref{eq:near_instance_optimality} due to the maximum over $s$ on the right hand side; it is very similar in spirit nonetheless.

Several Krylov methods for $f(A)\vb$ have been proposed as alternatives to the Lanczos method~\cite{ChenGreenbaumMuscoMusco2024,chen2023optimal,VanDenEshofFrommerLippertSchillingVanDerVorst2002,pleiss2020fast} which satisfy certain optimality guarantees (for restricted function classes and with respect to specific norms). Interestingly, they are typically outperformed in practice by the plain Lanczos method. Our analysis in \Cref{sec:near_optimality_stieltjes} further motivates why this observation is somewhat expected, proving that Lanczos does indeed satisfy a near optimality guarantee, at least for a rather large class of relevant functions $f$.

\section{Near instance optimality for Stieltjes functions}\label{sec:near_optimality_stieltjes}

Our main result is given in the following theorem.

\begin{theorem}\label{thm:near_instance_optimality_stieltjes}
Let $A \in \Cnn$ be Hermitian positive definite with smallest and largest eigenvalue $\lmin$ and $\lmax$, respectively, let $\vb \in \Cn$ with $\|\vb\| = 1$ and let $f$ be a Stieltjes function. Then the Lanczos approximation $\vf_m$ satisfies
\begin{align}
\|f(A)\vb-\vf_m\| &\leq \left(1+\beta_{m+1}\frac{\lmax}{\lmin^2}\right) \min_{p \in \Pi_{m-1}}\|f(A)\vb-p(A)\vb\| \label{eq:mainresult1}\\
                  &\leq \left(1+\kappa(A)^2\right) \min_{p \in \Pi_{m-1}}\|f(A)\vb-p(A)\vb\|. \label{eq:mainresult2}
\end{align}
\end{theorem}

In order to prove \Cref{thm:near_instance_optimality_stieltjes}, we require a few auxiliary results that we present next. In the following, we denote by $\vfmopt$ the optimal approximation for $f(A)\vb$ from $K_m(A,\vb)$ with respect to the Euclidean norm, i.e.,
\[
\|f(A)\vb - \vfmopt\| = \min_{\vx \in \spK_m(A,\vb)} \|f(A)\vb - \vx\| = \min_{p \in \Pi_{m-1}} \|f(A)\vb - p(A)\vb\|.
\]
Clearly, $\vfmopt$ corresponds to the orthogonal projection of $f(A)\vb$ onto $\spK_m(A,\vb)$, i.e., 
\begin{equation}\label{eq:vfmopt_projection}
    \vfmopt = V_mV_m^\ast f(A)\vb.
\end{equation}

The following proposition is an easy consequence of the finite termination property of the Lanczos method. %As before, we denote by $M$ the invariance index of $\spK_m(A,\vb)$.

\begin{proposition}\label{prop:block_partitioning}
Let $A$ be Hermitian positive definite, let $M$ be the invariance index of the Krylov subspace corresponding to $A$ and $\vb$ and let $V_M, T_M$ be the orthonormal basis and tridiagonal matrix resulting from $M$ iterations of \Cref{alg:lanczos}. For $m < M$, partition $V_M = [V_m, U_{M-m}]$ with $V_m \in \C^{n \times m}, U_{M-m} \in \C^{n \times (M-m)}$ and denote $f(T_m)\ve_1^{(m)} =: \vy_m$ and $f(T_M)\ve_1^{(M)} =: \begin{bmatrix} \vx_m \\ \vz_{M-m}\end{bmatrix}$ with $\vx_m \in \Cm, \vz_{M-m} \in \C^{M-m}$. Then
\begin{equation}\label{eq:err_lan}
f(A)\vb - \vf_m = V_m(\vx_m-\vy_m) + U_{M-m}\vz_{M-m} 
\end{equation}
and
\begin{align}
f(A)\vb - \vfmopt = U_{M-m}\vz_{M-m}.  \label{eq:err_opt}
\end{align}
\end{proposition}
\begin{proof}
Due to the finite termination property of the Lanczos method, we have $f(A)\vb = V_Mf(T_M)\ve_1^{(M)}$. We can therefore write
\begin{align*}
f(A)\vb - \vf_m &= V_Mf(T_M)\ve_1^{(M)} - V_mf(T_m)\ve_1^{(m)}\\
&= [V_m, U_{M-m}]\begin{bmatrix} \vx_m \\ \vz_{M-m}\end{bmatrix} - V_m\vy_m \\
&= V_m(\vx_m-\vy_m) + U_{M-m}\vz_{M-m}.
\end{align*}
Similarly, we find
\begin{align*}
f(A)\vb - \vfmopt &= V_Mf(T_M)\ve_1^{(M)} - V_mV_m^\ast V_Mf(T_M)\ve_1^{(M)}\\
&= [V_m, U_{M-m}]\begin{bmatrix} \vx_m \\ \vz_{M-m}\end{bmatrix} - V_mV_m^\ast[V_m, U_{M-m}]\begin{bmatrix} \vx_m \\ \vz_{M-m}\end{bmatrix} \nonumber\\
&= U_{M-m}\vz_{M-m},
\end{align*}
where the last equality follows because $V_m^\ast V_m = I_m$ and $V_m^\ast U_{M-m} = 0$.
\end{proof}

In the following, we partition $T_M$ in accordance with $V_M = [V_m, U_{M-m}]$, i.e.,
\begin{equation}\label{eq:TM}
T_M = \begin{bmatrix}
T_m & \beta_{m+1} \ve_m^{(m)}(\ve_1^{(M-m)})^\ast \\
\beta_{m+1} \ve_{M-m}^{(M-m)}(\ve_1^{(m)})^\ast & S_{M-m}
\end{bmatrix}.
\end{equation}

By considering $T_M$ in~\eqref{eq:TM} as a rank-two update of a block diagonal matrix and employing the Woodbury matrix identity, we can derive explicit formulas for the quantities $\vx_m-\vy_m$ and $\vz_{M-m}$ occurring in~\eqref{eq:err_lan}--\eqref{eq:err_opt}.

\begin{lemma}\label{lem:lowrank_update}
Let the assumptions of \Cref{prop:block_partitioning} hold, let $T_M$ be partitioned as in~\eqref{eq:TM} and let $f$ be a Stieltjes function of the form~\eqref{eq:stieltjes_function}. Define the scalar functions
\begin{align}\label{eq:auxiliary_functions}
\begin{split}
\gamma(t)      &= (\ve_m^{(m)})^\ast (T_m+tI)^{-1}\ve_m^{(m)},\\
\delta(t)      &= (\ve_1^{(M-m)})^\ast (S_{M-m}+tI)^{-1}\ve_1^{(M-m)},\\
\varepsilon(t) &= (\ve_m^{(m)})^\ast(T_m+tI)^{-1}\ve_1^{(m)}
\end{split}
\end{align}
and the matrix-valued function
\begin{equation}\label{eq:Xt}
X(t) = \begin{bmatrix} 
\gamma(t) & \frac{1}{\beta_{m+1}} \\ 
\frac{1}{\beta_{m+1}} & \delta(t)
\end{bmatrix} \in \C^{2 \times 2}.
\end{equation}
Then 
\[
\vx_m-\vy_m = f_1(T_m)\ve_m^{(m)}
\]
and
\[
\vz_{M-m} = f_2(S_{M-m})\ve_{M-m}^{(M-m)},
\]
where
\begin{equation*}
f_1(z) = \int_0^\infty -\frac{\delta(t)\varepsilon(t)}{\det(X(t))} \frac{1}{z+t} \dmu \ \text{ and } \
f_2(z) = \int_0^\infty \frac{\varepsilon(t)}{\beta_{m+1}\det(X(t))}\frac{1}{z+t} \dmu.
\end{equation*}
\end{lemma}
\begin{proof}
We mention upfront that the existence of the integrals in the definition of $f_1$ and $f_2$ will be assumed here. We prove in \Cref{lem:f1f2_stieltjes} below that they are indeed guaranteed to exist.

We note that for a Stieltjes function $f$, we have
\[
f(A)\vb = \int_0^\infty (A+tI)^{-1}\vb \dmu,
\]
and begin by focusing on an individual shifted inverse $(A+tI)^{-1}\vb$ for some $t \geq 0$. Analogously to the notation used in~\Cref{prop:block_partitioning}, we denote the coefficient vectors related to the Lanczos approximation of $(A+tI)^{-1}\vb$ by $\vy_m(t), \vx_m(t)$ and $\vz_{M-m}(t)$.

We define the block diagonal matrix
\[
D_M = \begin{bmatrix} T_m & \\ & S_{M-m}\end{bmatrix},
\]
with which we can write
\begin{equation}\label{eq:low_rank_update}
T_M = D_M + WRW^\ast, \text{ where } W = [\ve_m^{(M)}, \ve_{m+1}^{(M)}] \text{ and } R = \begin{bmatrix} 0 & \beta_{m+1} \\ \beta_{m+1} & 0\end{bmatrix}.
\end{equation}
By the Woodbury matrix identity~\cite{Woodbury1950}, we can express the shifted inverse of $T_M$ as
\begin{equation}\label{eq:woodbury}
(T_M+tI)^{-1} = (D_M + tI)^{-1} - (D_M + tI)^{-1}WX(t)^{-1}W^\ast(D_M+tI)^{-1},
\end{equation}
where 
\[
X(t) = (R^{-1} + W(D_M+tI)^{-1}W^\ast).
\]
Note that $X(t)$ is guaranteed to be invertible, because $T_M+tI, D_M+tI$ and $R$ are all invertible. 

By exploiting the block diagonal structure of $D_M+tI$ together with the sparsity pattern of $W$, we find
\begin{equation}\label{eq:DtIW}
(D_M+tI)^{-1}W = \begin{bmatrix} (T_m+tI)^{-1}\ve_m^{(m)} & \vnull^{(m)} \\ \vnull^{(M-m)} & (S_{M-m}+tI)^{-1}\ve_1^{(M-m)} \end{bmatrix},
\end{equation}
where $\vnull^{(m)} \in \Rm$ denotes a vector of all zeros. Multiplying~\eqref{eq:DtIW} by $W^\ast$ from the left and again exploiting the zero pattern, a direct computation shows that $X(t)$ can be written as in~\eqref{eq:Xt}. Therefore, the matrix $X(t)^{-1}$ occurring in~\eqref{eq:woodbury} is given by
\begin{equation}\label{eq:Xtinv}
X(t)^{-1} = \frac{1}{\det(X(t))} \begin{bmatrix} \delta(t) & -\frac{1}{\beta_{m+1}} \\ -\frac{1}{\beta_{m+1}} & \gamma(t) \end{bmatrix} \ \text{with}\ \det(X(t)) = \gamma(t)\delta(t) - \frac{1}{\beta_{m+1}^2}.
\end{equation}
Inserting~\eqref{eq:DtIW} and~\eqref{eq:Xtinv} into~\eqref{eq:woodbury} yields
\begin{equation}\label{eq:update}
(T_M+tI)^{-1} = (D_M+tI)^{-1} - \frac{1}{\det(X(t))} N(t)
\end{equation}
where 
\[
N(t) = \begin{bmatrix}
   N_{11}(t) & N_{12}(t) \\
   N_{21}(t) & N_{22}(t)
\end{bmatrix}
\]
has the blocks
\begin{align*}
N_{11}(t) &:= \delta(t) (T_m+tI)^{-1}\ve_m^{(m))}(\ve_m^{(m)})^\ast(T_m+tI)^{-1}, \\ 
N_{12}(t) &:= -\frac{1}{\beta_{m+1}}(T_m+tI)^{-1}\ve_m^{(m)}(\ve_1^{(M-m))})^\ast(S_{M-m}+tI)^{-1}, \\
N_{21}(t) &:= -\frac{1}{\beta_{m+1}}(S_{M-m}+tI)^{-1}\ve_1^{(M-m)}(\ve_m^{(m))})^\ast(T_m+tI)^{-1}, \\
N_{22}(t) &:= \gamma(t) (S_{M-m}+tI)^{-1}\ve_1^{(M-m)}(\ve_1^{(M-m))})^\ast(S_{M-m}+tI)^{-1}.
\end{align*}
We have
\begin{align*}
(T_M+tI)^{-1}\ve_1^{(M)} - (D_M+tI)^{-1}\ve_1^{(M)} &= \begin{bmatrix} \vx_m(t) \\ \vz_{M-m}(t) \end{bmatrix}- \begin{bmatrix} \vy_m(t) \\ \vnull^{(M-m)} \end{bmatrix}\\
&= \begin{bmatrix} \vx_m(t)-\vy_m(t) \\ \vz_{M-m}(t) \end{bmatrix},
\end{align*}
which by~\eqref{eq:update} implies
\[
\begin{bmatrix} \vx_m(t)-\vy_m(t) \\ \vz_{M-m}(t) \end{bmatrix} = - \frac{1}{\det(X(t))} N(t)\ve_1^{(M)},
\]
i.e.,
\begin{equation}\label{eq:xmtymt}
\vx_m(t)-\vy_m(t) = - \frac{\delta(t)\varepsilon(t)}{\det(X(t))} (T_m+tI)^{-1}\ve_m^{(m)}
\end{equation}
and
\begin{equation}\label{eq:zmt}
\vz_{M-m}(t) = \frac{\varepsilon(t)}{\beta_{m+1}\det(X(t))}(S_{M-m}+tI)^{-1}\ve_1^{(M-m)}
\end{equation}
with $\varepsilon(t)$ defined in~\eqref{eq:auxiliary_functions}. 

The assertion of the lemma follows from~\eqref{eq:xmtymt} and~\eqref{eq:zmt} by noting that
\[
\vx_m = \int_0^\infty \vx_m(t)\dmu, \ \vy_m = \int_0^\infty \vy_m(t)\dmu \ \text{ and } \ \vz_{M-m} = \int_0^\infty \vz_{M-m}(t)\dmu.
\]
\end{proof}

Our next auxiliary result states that the functions$f_1$ and $f_2$ defined in \Cref{lem:lowrank_update} are (scalar multiples of) Stieltjes functions. This is important because it not only guarantees that they are well-defined for any $z \in (0, \infty)$, but by \Cref{prop:stieltjes_monotonic} it also implies that they are monotonically decreasing (in magnitude) on $(0, \infty)$, which is an essential argument in the proof of \Cref{thm:near_instance_optimality_stieltjes}.

\begin{lemma}\label{lem:f1f2_stieltjes}
Let the assumptions of \Cref{lem:lowrank_update} hold. Then $(-1)^{m+1}f_1$ and $(-1)^mf_2$ are Stieltjes functions. In particular, $f_1, f_2$ have constant sign on $(0, \infty)$ and $|f_1|, |f_2|$ are monotonically decreasing on $(0, \infty)$.
\end{lemma}
\begin{proof}
We begin by proving a few auxiliary results about the properties of the involved functions $\gamma(t), \delta(t), \varepsilon(t)$ and $\det(X(t))$.

As $T_m = V_m^\ast A V_m$, $S_{M-m} = U_{M-m}^\ast A U_{M-m}$, we have $\spec(T_m)$, $\spec(S_{M-m}) \subset [\lmin,\lmax]$, and therefore $\spec((T_m+tI)^{-1})$, $\spec((S_{M-m}+tI)^{-1}) \subset [\frac{1}{\lmax+t},\frac{1}{\lmin+t}]$. Thus, the Rayleigh quotients $\gamma(t), \delta(t)$ satisfy
\begin{equation}\label{eq:bounds_gamma_delta}
\frac{1}{\lmax+t} \leq \gamma(t) \leq \frac{1}{\lmin+t} \quad \text{ and } \quad\frac{1}{\lmax+t} \leq \delta(t) \leq \frac{1}{\lmin+t}.
\end{equation}
In particular, $\gamma(t)$ and $\delta(t)$ are positive for all $t \geq 0$ and $\max_{t \geq 0} \gamma(t) \leq \frac{1}{\lmin}$, $\max_{t \geq 0} \delta(t) \leq \frac{1}{\lmin}$.

For investigating $\varepsilon(t)$, we exploit standard properties of tridiagonal matrices to write
\[
\varepsilon(t) = (\ve_m^{(m)})^\ast(T_m+tI)^{-1}\ve_1^{(m)} = (-1)^{m+1}\frac{\prod_{i=1}^{m-1} \beta_i}{\prod_{i=1}^m \theta_i+t},
\]
where $\theta_1,\dots,\theta_m \subset [\lmin, \lmax]$ denote the Ritz values (i.e., the eigenvalues of $T_m$). As all $\beta_i, \theta_i > 0$, it is immediate that $(-1)^{m+1}\varepsilon(t)$ is positive and monotonically decreasing in $t$. Lastly, consider 
\[
\det(X(t)) = \gamma(t)\delta(t) - \frac{1}{\beta_{m+1}^2}.
\]
Similarly to $\delta(t)$, we know that $\gamma(t)$ is positive and bounded above, $\max_{t \geq 0} \gamma(t) \leq \frac{1}{\lmin}$, as it is a Rayleigh quotient of $(T_m+tI)^{-1}$. 
It is easy to see that the functions $\gamma(t),\delta(t)$ are monotonically decreasing in $t$ and satisfy $\gamma(t), \delta(t) \rightarrow 0$ for $t \rightarrow \infty$, so that we can conclude that $\det(X(t)) \rightarrow -\frac{1}{\beta_{m+1}^2} < 0$. As $\det(X(t))$ depends continuously on $t$ and $X(t)$ is invertible for all $t \geq 0$, this implies that $\det(X(t)) < 0$ for all $t \geq 0$. In particular, $\det(X(0)) = \gamma(0)\delta(0) - \frac{1}{\beta_{m+1}^2} < 0$ and we have the bound
\[
-\frac{1}{\beta_{m+1}^2} \leq \det(X(t)) \leq \gamma(0)\delta(0) - \frac{1}{\beta_{m+1}^2} < 0 \text{ for all } t\geq 0.
\]
In summary, we can conclude that $(-1)^{m+1}\frac{\delta(t)\varepsilon(t)}{\det(X(t))}$ and $(-1)^m\frac{\varepsilon(t)}{\beta_{m+1}\det(X(t))}$ are positive and go to zero at least as fast as $\frac{1}{1+t}$. Thus, $(-1)^{m+1}f_1, (-1)^mf_2$ fulfill the conditions of \Cref{prop:stieltjes_modified_measure} and are therefore Stieltjes functions.
\end{proof}

With these preparations, we are now in position to prove our main result.

\begin{proof}[Proof of \Cref{thm:near_instance_optimality_stieltjes}]
Throughout this proof, we use the notations established in \Cref{prop:block_partitioning}, \Cref{lem:lowrank_update} and \Cref{lem:f1f2_stieltjes}.

Due to the unitary invariance of the Euclidean norm, we directly obtain
\begin{equation}\label{eq:norm_z}
\|f(A)\vb - \vfmopt\| = \|U_{M-m}\vz_{M-m}\| = \|\vz_{M-m}\|
\end{equation}
from \Cref{prop:block_partitioning}. Using the triangle inequality together with the unitary invariance, we further have
\begin{align}
\|f(A)\vb - \vf_m\| &\leq \|V_m(\vx_m-\vy_m)\| + \|U_{M-m}\vz_{M-m}\| \nonumber\\
                    &= \|\vx_m-\vy_m\| + \|\vz_{M-m}\| \nonumber\\
                    &= \left(1+\frac{\|\vx_m-\vy_m\|}{\|\vz_{M-m}\|}\right) \|f(A)\vb-\vfmopt\|,\label{eq:bound_err_lan}
\end{align}
where we used~\eqref{eq:norm_z} for the last equality. From \Cref{lem:lowrank_update}, we obtain that 
\begin{equation}\label{eq:bound_ratio1}
\frac{\|\vx_m-\vy_m\|}{\|\vz_{M-m}\|} = \frac{\|f_1(T_m)\ve_m^{(m)}\|}{\|f_2(S_{M-m})\ve_1^{(M-m)}\|} \leq \frac{\max_{\lambda \in \spec(T_m)} |f_1(\lambda)|}{\min_{\lambda \in \spec(S_{M-m})} |f_2(\lambda)|}.
\end{equation}
As $|f_1|,|f_2|$ are monotonically decreasing on $(0,\infty)$ and $\spec(T_m)$, $\spec(S_{M-m}) \subset [\lmin,\lmax]$, we find $\max_{\lambda \in \spec(T_m)} |f_1(\lambda)| \leq |f_1(\lmin)|$ and $\min_{\lambda \in \spec(S_{M-m})} |f_2(\lambda)| \geq |f_2(\lmax)|$. Using these bounds,~\eqref{eq:bound_ratio1} implies
\begin{equation}\label{eq:bound_ratio2}
\frac{\|\vx_m-\vy_m\|}{\|\vz_{M-m}\|} \leq \frac{|f_1(\lmin)|}{|f_2(\lmax)|}.
\end{equation}
From~\eqref{eq:bounds_gamma_delta} we have $\delta(t) \leq \frac{1}{\lmin}$, and further $\frac{1}{\lmin+t} \leq \frac{\lmax}{\lmin}\frac{1}{\lmax+t}$ holds for all $t \geq 0$. Using these facts, we can write
\begin{align}
|f_1(\lmin)| &= \int_0^\infty \frac{|\delta(t)| |\varepsilon(t)|}{|\det(X(t))|} \frac{1}{\lmin+t} \dmu \nonumber\\ 
&\leq \frac{\beta_{m+1}\delta(0)\lmax}{\lmin} \int_0^\infty \frac{|\varepsilon(t)|}{\beta_{m+1}|\det(X(t))|} \frac{1}{\lmax+t} \dmu \label{eq:bound_f1f2_intermediate}\\
&= \frac{\beta_{m+1}\lmax}{\lmin^2} |f_2(\lmax)|. \label{eq:bound_f1f2}
\end{align}
Note that in the first and last equality, we exploited that all terms in the integrand have constant sign, so that 
\[
|f_1(z)| = \int_0^\infty \left|\frac{\delta(t)\varepsilon(t)}{\det(X(t))} \frac{1}{z+t}\right| \dmu, \  \
|f_2(z)| = \int_0^\infty \left|\frac{\varepsilon(t)}{\beta_{m+1}\det(X(t))}\frac{1}{z+t}\right| \dmu.\footnote{It appears to be mainly this step which makes it difficult to straightforwardly generalize our result to other function classes such as ``Stieltjes'' functions corresponding to a signed measure $\d\mu$ or general analytic functions represented by the Cauchy integral formula.}
\]
Inserting~\eqref{eq:bound_f1f2} into~\eqref{eq:bound_ratio2} proves~\eqref{eq:mainresult1}. The inequality~\eqref{eq:mainresult2} directly follows by noting that $\beta_{m+1} \leq \|T_m\| \leq \lmax$ and $\kappa(A) = \frac{\lmax}{\lmin}$.
\end{proof}

\begin{remark}\label{rem:beta}
The bound~\eqref{eq:mainresult2} is likely a large overestimate of the actual behavior. In particular, in~\eqref{eq:bound_ratio1} we applied a very rough estimate by upper bounding the numerator and lower bounding the denominator, while in reality one often observes that both are of roughly equal magnitude; see also \Cref{ex:sharpness} below. Another rough estimate occurs when going from~\eqref{eq:mainresult1} to~\eqref{eq:mainresult2} by bounding $\beta_{m+1} \leq \lmax$, in particular taking into account that $\beta_{m+1}$ typically decreases once $\spK_m(A,\vb)$ becomes close to an invariant subspace. If trying to bound the distance to the optimal error in an actual computation, one can omit this estimate, as $\beta_{m+1}$ is readily available from the Lanczos process. Another nice feature of keeping $\beta_{m+1}$ in the bound is that it reflects that upon a lucky breakdown---in which case we have $\beta_{m+1} = 0$---the Lanczos approximation is equal to the optimal approximation from the Krylov space which in this case is $f(A)\vb$.
\end{remark}

\begin{remark}\label{rem:active_eigenvalues}
The bounds in \Cref{thm:near_instance_optimality_stieltjes} can be refined if it is known that the vector $\vb$ only contains contributions from certain eigenvectors of $A$. Let us write $\vb = \sum_{i = 1}^n c_i \vw_i$, where $\vw_i, i = 1,\dots,n$ are the orthonormal eigenvectors of $A$ corresponding to $\lmin = \lambda_1 \leq \lambda_2 \leq \dots \leq \lambda_{n-1} \leq \lambda_n = \lmax$ of $A$. If $c_i = 0$ for $i < j$ and if $c_i = 0$ for $i > k$, then $\lmin$ can be replaced by $\lambda_j$ and $\lmax$ can be replaced by $\lambda_k$, as $\spec(T_m), \spec(S_{M-m}) \subset [\lambda_j, \lambda_k]$ in that case; also see \Cref{ex:effective_condition_number} below.
\end{remark}

\section{Extension to a related function class}\label{sec:related_classes}
Another relevant class of functions, which is intimately related to Stieltjes functions, is given by functions of the form $f(z) = zg(z)$, where $g$ is a Stieltjes function. Practically relevant examples of functions of this type are the square root $f(z) = \sqrt{z}$ and the shifted logarithm $f(z) = \log(z+1)$. The result of \Cref{thm:near_instance_optimality_stieltjes} can straight-forwardly be extended to functions of this class, and we only sketch the corresponding proof.

\begin{theorem}\label{thm:extension}
The statement of \Cref{thm:near_instance_optimality_stieltjes} remains valid if $f(z) = zg(z)$, where $g$ is a Stieltjes function. 
\end{theorem}
\begin{proof}
In this modified setting, an analogous version of \Cref{lem:lowrank_update} holds, where $f_1, f_2$ are replaced by the functions $\widetilde{f}_1(z) = zf_1(z)$ and $\widetilde{f}_2(z) = zf_2(z)$. According to \Cref{prop:zf_increasing}, $|\widetilde{f}_1|, |\widetilde{f}_2|$ are monotonically \emph{increasing} on $(0,\infty)$. Therefore, by following the same steps as in the proof of \Cref{thm:near_instance_optimality_stieltjes}, instead of~\eqref{eq:bound_ratio2}, we find
\begin{equation}
\frac{\|\vx_m-\vy_m\|}{\|\vz_{M-m}\|} \leq \frac{|\widetilde{f}_1(\lmax)|}{|\widetilde{f}_2(\lmin)|}.
\end{equation}
Proceeding in a similar manner as before, this time using $\frac{1}{\lmax+t} \leq \frac{1}{\lmin+t}$ for all $t \geq 0$, we obtain
\begin{align}
|\widetilde{f}_1(\lmax)| &= \lmax \int_0^\infty \frac{|\delta(t)| |\varepsilon(t)|}{|\det(X(t))|} \frac{1}{\lmax+t} \dmu \nonumber\\ 
&\leq \frac{\beta_{m+1}\lmax}{\lmin^2} \lmin\int_0^\infty \frac{|\varepsilon(t)|}{\beta_{m+1}|\det(X(t))|} \frac{1}{\lmin+t} \dmu \nonumber\\
&= \frac{\beta_{m+1}\lmax}{\lmin^2} |\widetilde{f}_2(\lmin)|,
\end{align}
from which the result follows.
\end{proof}

Of course, the comments made in \Cref{rem:beta,rem:active_eigenvalues} also remain valid for \Cref{thm:extension}.

\section{Numerical examples}\label{sec:examples}
In this section, we illustrate our theoretical results by some examples. Note that the experiments in this section are performed on simple toy problems (diagonal $A$ of small size), as properly evaluating the bounds requires knowledge of the exact solution $f(A)\vb$, the optimal approximation $\vfmopt$ and possibly all eigenvalues of $A$, quantities that are not available for large-scale real-world problems. 

All experiments are performed in MATLAB 2024b.
%This is appropriate as our results are mostly of theoretical nature. The practical performance of the Lanczos method for $f(A)\vb$ for large-scale problems is well-studied in many algorithm-oriented papers and not the focus of our work.

%The examples in this section serve two main purposes. First, we illustrate the sharpness of the two bounds~\eqref{eq:mainresult1}--\eqref{eq:mainresult2} from \Cref{thm:near_instance_optimality_stieltjes}, together with the ``intermediate'' bounds\eqref{eq:bound_ratio1} and~\eqref{eq:bound_f1f2_intermediate} appearing in the proof of the theorem. This serves the purpose of understanding which estimates occurring in the proof cause the loss of sharpness in the bound; cf.~also \Cref{rem:beta}. 

%In addition, we compare the bounds arising from \Cref{thm:near_instance_optimality_stieltjes,thm:extension} to previously available results, i.e., to bounds from~\cite{AmselChenGreenbaumMuscoMusco2023} and the well-known near FOV optimality bound~\eqref{eq:near_fov_optimality}. 

\begin{example}\label{ex:sharpness}
\begin{figure}
    \centering
    \pgfplotsset{height=0.43\linewidth,width=0.48\linewidth,compat=1.10,every axis/.append style={legend style={/tikz/every even column/.append style={column sep=6pt}}}}
\pgfplotsset{every tick label/.append style={font=\scriptsize}}

\noindent%
\begin{tikzpicture}[scale=1]%
\begin{groupplot}[group style={group size=2 by 2, horizontal sep=2cm, vertical sep=1.5cm}]

\nextgroupplot[legend columns=3, legend style={at={(.4,1.05)},anchor=south west}, cycle list name=list_bounds, grid=major, 
   	xlabel={\scriptsize Krylov dimension $m$}, ylabel={\scriptsize Error norm}, ymode = log, xmin = 0, xmax = 65]

\addplot+[mark repeat=10] table [x ={n},y ={opt}] {fig/invsqrt_bound_sharpness.dat};\addlegendentry{\scriptsize  Optimal}
\addplot+[mark repeat=10] table [x ={n},y ={lan}] {fig/invsqrt_bound_sharpness.dat};\addlegendentry{\scriptsize Lanczos}
\addplot+[] table [x ={n},y ={b2}] {fig/invsqrt_bound_sharpness.dat};\addlegendentry{\scriptsize Bound~\eqref{eq:mainresult1}}
\addplot+[] table [x ={n},y ={b1}] {fig/invsqrt_bound_sharpness.dat};\addlegendentry{\scriptsize Bound~\eqref{eq:mainresult2}}
\addplot+[] table [x ={n},y ={b4}] {fig/invsqrt_bound_sharpness.dat};\addlegendentry{\scriptsize Bound~\eqref{eq:intermediate1}}
\addplot+[] table [x ={n},y ={b3}] {fig/invsqrt_bound_sharpness.dat};\addlegendentry{\scriptsize Bound~\eqref{eq:intermediate2}}

\nextgroupplot[cycle list name=list_bounds, grid=major, 
   	xlabel={\scriptsize Krylov dimension $m$}, ylabel={\scriptsize  Error norm}, ymode = log, xmin = 0, xmax=30]

\addplot+[mark repeat=10] table [x ={n},y ={opt}] {fig/invsqrt_bound_sharpness2.dat};
\addplot+[mark repeat=10] table [x ={n},y ={lan}] {fig/invsqrt_bound_sharpness2.dat};
\addplot+[] table [x ={n},y ={b2}] {fig/invsqrt_bound_sharpness2.dat};
\addplot+[] table [x ={n},y ={b1}] {fig/invsqrt_bound_sharpness2.dat};
\addplot+[] table [x ={n},y ={b4}] {fig/invsqrt_bound_sharpness2.dat};
\addplot+[] table [x ={n},y ={b3}] {fig/invsqrt_bound_sharpness2.dat};

\nextgroupplot[cycle list name=list_bounds, grid=major, 
   	xlabel={\scriptsize Krylov dimension $m$}, ylabel={\scriptsize  Error norm}, ymode = log, xmin = 0, xmax=66]

\addplot+[mark repeat=10] table [x ={n},y ={opt}] {fig/sqrt_bound_sharpness.dat};
\addplot+[mark repeat=10] table [x ={n},y ={lan}] {fig/sqrt_bound_sharpness.dat};
\addplot+[] table [x ={n},y ={b2}] {fig/sqrt_bound_sharpness.dat};
\addplot+[] table [x ={n},y ={b1}] {fig/sqrt_bound_sharpness.dat};
\addplot+[] table [x ={n},y ={b4}] {fig/sqrt_bound_sharpness.dat};
\addplot+[] table [x ={n},y ={b3}] {fig/sqrt_bound_sharpness.dat};

\nextgroupplot[cycle list name=list_bounds, grid=major, 
   	xlabel={\scriptsize Krylov dimension $m$}, ylabel={\scriptsize  Error norm}, ymode = log, xmin = 0, xmax=27]

\addplot+[mark repeat=10] table [x ={n},y ={opt}] {fig/sqrt_bound_sharpness2.dat};
\addplot+[mark repeat=10] table [x ={n},y ={lan}] {fig/sqrt_bound_sharpness2.dat};
\addplot+[] table [x ={n},y ={b2}] {fig/sqrt_bound_sharpness2.dat};
\addplot+[] table [x ={n},y ={b1}] {fig/sqrt_bound_sharpness2.dat};
\addplot+[] table [x ={n},y ={b4}] {fig/sqrt_bound_sharpness2.dat};
\addplot+[] table [x ={n},y ={b3}] {fig/sqrt_bound_sharpness2.dat};

\end{groupplot}

\end{tikzpicture}
    
    \vspace{-.5cm}
    \caption{Sharpness of the estimates from~\Cref{thm:near_instance_optimality_stieltjes} as well as of certain inequalities from its proof for the matrices $A_1$ (left) and $A_2$ (right) defined in the text of \Cref{ex:sharpness}. The vector $\vb$ has normally distributed random entries and $f$ is the inverse square root (top row) or square root (bottom row).}
    \label{fig:sharpness1}
\end{figure}
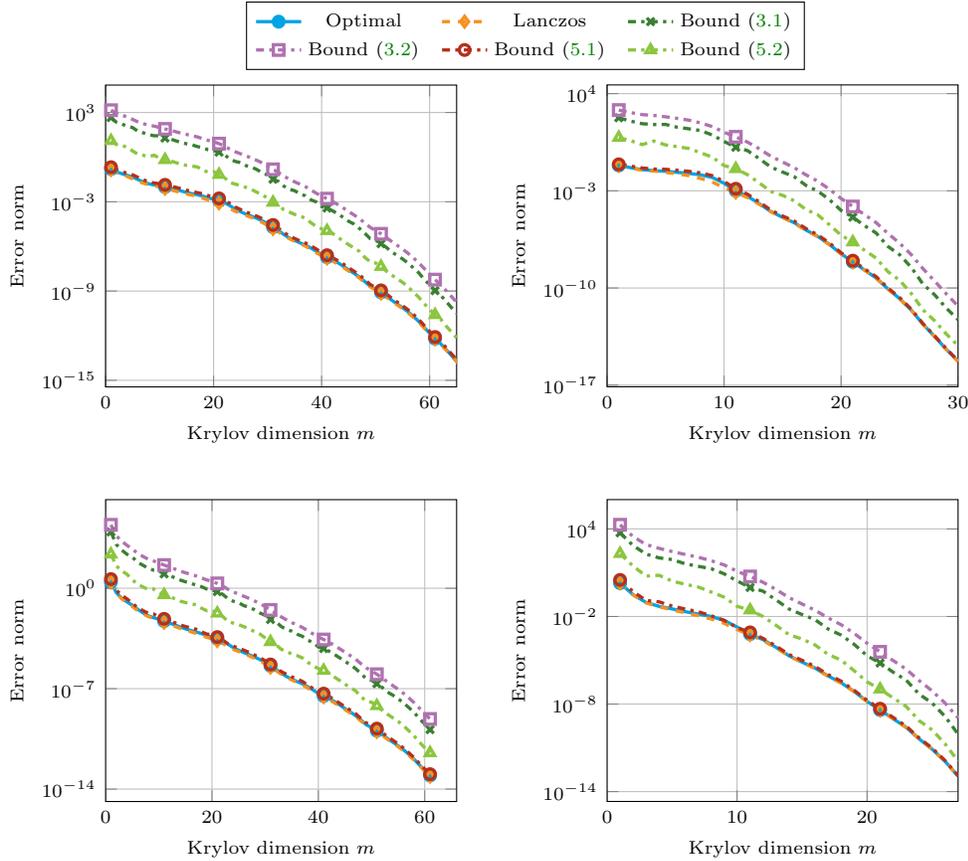

We begin by assessing the sharpness of the bounds and individual estimates. For this, we consider two test matrices inspired by the experiments reported in~\cite{AmselChenGreenbaumMuscoMusco2023}, $A_1 = \diag(1, 2, \dots,100) \in \R^{100 \times 100}$ and $A_2 = \diag(\eta_1,\dots,\eta_{100})  \in \R^{100 \times 100}$, where
\[
\eta_i = \left(1 + 99\left(\frac{1-\rho^{\frac{i-1}{99}}}{1-\rho}\right)\right), \qquad i = 1,\dots, 100,
\]
with $\rho = 0.001$. Obviously, $\kappa(A_1) = \kappa(A_2) = 100$. We construct $\vb \in \R^{100}$ with normally distributed random entries and scale it such that $\|\vb\| = 1$. \Cref{fig:sharpness1} shows our results for these matrices when $f$ is the inverse square root or square root. We observe that in all cases, the convergence curve of the Lanczos method is almost indistinguishable from that of the optimal Krylov approximation. In fact, extensive numerical evidence suggests that the Lanczos method often performs close to optimal also for other problems, so that the ``actual'' constant in~\eqref{eq:near_instance_optimality} probably satisfies $C = \mathcal{O}(1)$, at least for ``well-behaved'' functions like Stieltjes functions. 

\begin{figure}
    \centering
    \pgfplotsset{height=0.43\linewidth,width=0.48\linewidth,compat=1.10,every axis/.append style={legend style={/tikz/every even column/.append style={column sep=6pt}}}}
\pgfplotsset{every tick label/.append style={font=\scriptsize}}

\noindent%
\begin{tikzpicture}[scale=1]%
\begin{groupplot}[group style={group size=2 by 2, horizontal sep=2cm, vertical sep=.25cm}]

\nextgroupplot[cycle list name=list_bounds, grid=major, ylabel={\scriptsize norm}, ymode = log, xmin = 0, xmax = 65, legend pos = south west, xticklabel=\empty]

\addplot+[] table [x ={n},y ={f1T}] {fig/invsqrt_bound_ratio.dat};\addlegendentry{\scriptsize $\|f_1(T_m)\ve_m\|$}
\addplot+[] table [x ={n},y ={f2S}] {fig/invsqrt_bound_ratio.dat};\addlegendentry{\scriptsize  $\|f_2(S_{M-m})\ve_1\|$}

\nextgroupplot[cycle list name=list_bounds, grid=major, ylabel={\scriptsize  norm}, ymode = log, xmin = 0, xmax=30, xticklabel=\empty]

\addplot+[] table [x ={n},y ={f1T}] {fig/invsqrt_bound_ratio2.dat};
\addplot+[] table [x ={n},y ={f2S}] {fig/invsqrt_bound_ratio2.dat};

\nextgroupplot[height=.2\linewidth, grid=major, xlabel={\scriptsize Krylov dimension $m$}, ylabel={\scriptsize  ratio}, xmin = 0, xmax=65]

\addplot+[black, mark = none, line width=1.25pt] table [x ={n},y ={ratio}] {fig/invsqrt_bound_ratio.dat};

\nextgroupplot[height=.2\linewidth, grid=major, xlabel={\scriptsize Krylov dimension $m$}, ylabel={\scriptsize  ratio}, xmin = 0, xmax=30]

\addplot+[black, mark = none, line width=1.25pt] table [x ={n},y ={ratio}] {fig/invsqrt_bound_ratio2.dat};

\end{groupplot}

\end{tikzpicture}
    
    \vspace{-.5cm}
    \caption{Comparison of the norms of the two terms $\|f_1(T_m)\ve_m\|$ and $\|f_2(S_{M-m})\ve_1\|$ contributing to the Lanczos error for the matrices $A_1$ (left) and $A_2$ (right) defined in the text of \Cref{ex:sharpness}. The vector $\vb$ has normally distributed random entries and $f$ is the inverse square root. The bottom panel shows the ratio between the two terms.}
    \label{fig:ratio}
\end{figure}
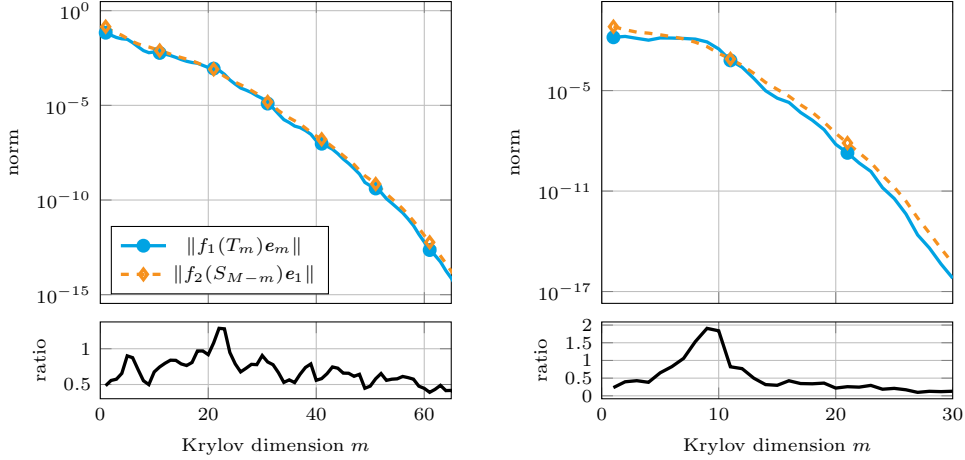

Therefore, the bound~\eqref{eq:mainresult2} of \Cref{thm:near_instance_optimality_stieltjes} largely overestimates the actual ratio between the Lanczos and optimal error, as we expected (see also \Cref{rem:beta}). Bound~\eqref{eq:mainresult1} is sharper, but still an overestimate. We also plot the values of the ``intermediate'' bounds
\begin{equation}\label{eq:intermediate1}
\|f(A)\vb-\vf_m\| \leq \left(1+\frac{\|f_1(T_m)\ve_m^{(m)}\|}{\|f_2(S_{M-m})\ve_1^{(M-m)}\|}\right)\|f(A)\vb-\vfmopt\|
\end{equation}
and
\begin{equation}\label{eq:intermediate2}
\|f(A)\vb-\vf_m\| \leq \left(1+\beta_{m+1}\delta(0)\kappa(A)\|\right)\|f(A)\vb-\vfmopt\|
\end{equation}
arising from the proof of \Cref{thm:near_instance_optimality_stieltjes} in order to illustrate which estimates in the proof cause the loss of sharpness in the bound. For evaluating~\eqref{eq:intermediate1}, the integrals in the definition of $f_1$ and $f_2$ are approximated roughly up to machine precision using the built-in MATLAB function \texttt{integral}.

We observe that~\eqref{eq:intermediate2} is already a lot sharper than~\eqref{eq:mainresult1}, suggesting that $\delta(0) \leq \frac{1}{\lmin}$ is a rather loose estimate. The estimate~\eqref{eq:intermediate1} bounds the actual Lanczos error extremely closely. This is to be expected, as the only slack in this bound comes from the use of the triangle inequality in~\eqref{eq:bound_err_lan}. As the two error components $V_m(\vx_m-\vy_m)$ and $U_{M-m}\vz_{M-m}$ are orthogonal to each other, they actually satisfy
\[
\|f(A)\vb - \vf_m\|^2 = \|V_m(\vx_m-\vy_m)\|^2 + \|U_{M-m}\vz_{M-m}\|^2 = \|\vx_m-\vy_m\|^2 + \|\vz_{M-m}\|^2
\]
by the Pythagorean theorem. Therefore, the triangle inequality introduces a relative slack of at most $\sqrt{2}$. Of course, the prefactors in~\eqref{eq:intermediate1} and~\eqref{eq:intermediate2} cannot be practically computed, as they depend on the unknown matrix $S_{M-m}$ which would only be available upon running the Lanczos method until termination (in~\eqref{eq:intermediate2}, this dependence is hidden inside the function $\delta(t)$).

To further exemplify that bounding $\frac{\|f_1(T_m)\ve_m^{(m)}\|}{\|f_2(S_{M-m})\ve_1^{(M-m)}\|}$ is the main reason for the unnecessary increase in the constant $C$ of our near instance optimality bound, we plot the two norms (as well as their ratio) across all iterations in \Cref{fig:ratio}. Both norms decay at about the same rate as the iteration progresses and their ratio constantly lies in the interval $[0.5, 2]$, indicating that $C = \mathcal{O}(1)$ is indeed a reasonable conjecture. This time, we only report results for the inverse square root, as results for the square root (as well as results of many other experiments not reported here) look strikingly similar.
\end{example}

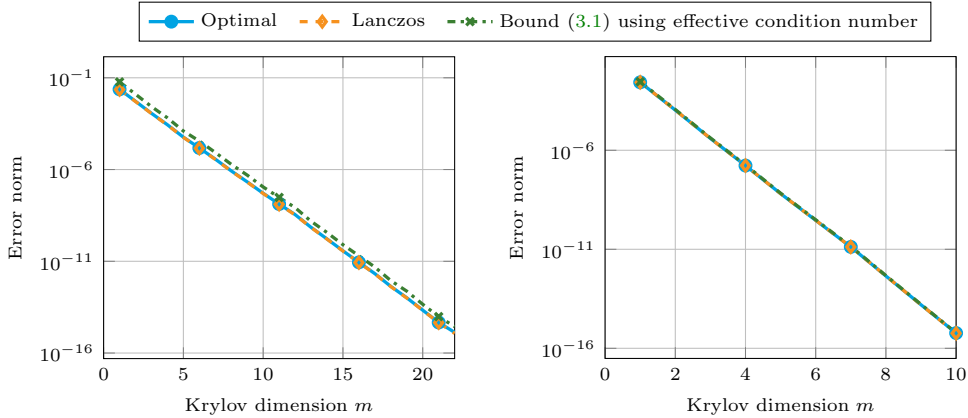
\begin{figure}
    \centering
    \pgfplotsset{height=0.43\linewidth,width=0.48\linewidth,compat=1.10,every axis/.append style={legend style={/tikz/every even column/.append style={column sep=6pt}}}}
\pgfplotsset{every tick label/.append style={font=\scriptsize}}

\noindent%
\begin{tikzpicture}[scale=1]%
\begin{groupplot}[group style={group size=2 by 1, horizontal sep=2cm, vertical sep=2cm}]

\nextgroupplot[legend columns=3, legend style={at={(.1,1.05)},anchor=south west}, cycle list name=list_bounds, grid=major, xlabel={\scriptsize Krylov dimension $m$}, ylabel={\scriptsize Error norm}, ymode = log, xmin = 0, xmax = 22]

\addplot+[mark repeat=5] table [x ={n},y ={opt}] {fig/invsqrt_bound_effective.dat};\addlegendentry{\scriptsize  Optimal}
\addplot+[mark repeat=5] table [x ={n},y ={lan}] {fig/invsqrt_bound_effective.dat};\addlegendentry{\scriptsize Lanczos}
\addplot+[] table [x ={n},y ={b}] {fig/invsqrt_bound_effective.dat};\addlegendentry{\scriptsize Bound~\eqref{eq:mainresult1} using effective condition number}

\nextgroupplot[cycle list name=list_bounds, grid=major, 
   	xlabel={\scriptsize Krylov dimension $m$}, ylabel={\scriptsize  Error norm}, ymode = log, xmin = 0, xmax=10]

\addplot+[mark repeat=3] table [x ={n},y ={opt}] {fig/invsqrt_bound_effective2.dat};
\addplot+[mark repeat=3] table [x ={n},y ={lan}] {fig/invsqrt_bound_effective2.dat};
\addplot+[] table [x ={n},y ={b}] {fig/invsqrt_bound_effective2.dat};

\end{groupplot}

\end{tikzpicture}

    \vspace{-.5cm}
    \caption{Effective bound from~\Cref{thm:near_instance_optimality_stieltjes} for the matrices $A_1$ (left) and $A_2$ (right) defined in the text of \Cref{ex:sharpness}. The function $f$ is the inverse square root and the vector $\vb$ has normally distributed contribution from the eigenvectors $\vw_{26},\dots,\vw_{75}$, while the other eigenvectors have zero contribution. Thus, in~\eqref{eq:mainresult1}, we replace $\lmin$ by $\lambda_{26}$ and $\lmax$ by $\lambda_{75}$.}
    \label{fig:effective}
\end{figure}

\begin{example}\label{ex:effective_condition_number}
Next, we use almost the same setup as in \Cref{ex:sharpness}, but modify the vector $\vb$ such that it only contains contributions from some eigenvectors. This serves the purpose of illustrating the statement of \Cref{rem:active_eigenvalues}. As $A_1, A_2$ are diagonal, an orthonormal basis of eigenvectors is $\ve_1^{(n)}, \dots,\ve_n^{(n)}$ and the corresponding eigenvalues are ordered ascendingly, $\lmin = \lambda_1 \leq \lambda_2 \leq \dots \leq \lambda_n = \lmax$. We let $\vb = \sum_{i=26}^{75} c_i\ve_i^{(n)}$ where $c_i$ are normally distributed (and scaled such that $\vb$ has unit norm). In this situation, the extremal eigenvalues of $A$ in the bounds of \Cref{thm:near_instance_optimality_stieltjes} can be replaced by their effective counterparts $\lambda_{26}$ and $\lambda_{75}$, so that the constant in our near-optimality result becomes much smaller. The corresponding results are depicted in \Cref{fig:effective}. We again only report results for the inverse square root function, as the results for the square root are very similar. As expected, convergence becomes more rapid when $\vb$ contains only contributions from some part of the eigenvectors of $A$, and our bound~\eqref{eq:mainresult1} very tightly follows the actual error due to the small effective condition number.
\end{example}

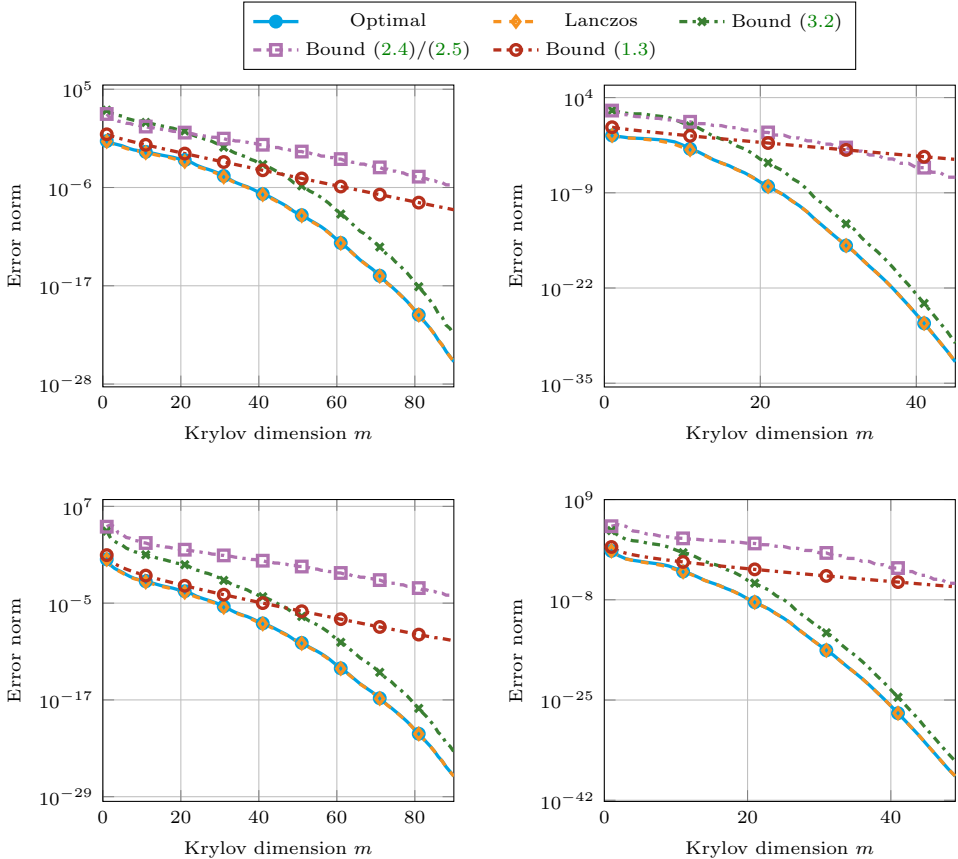
\begin{figure}
    \centering
    \pgfplotsset{height=0.43\linewidth,width=0.48\linewidth,compat=1.10,every axis/.append style={legend style={/tikz/every even column/.append style={column sep=6pt}}}}
\pgfplotsset{every tick label/.append style={font=\scriptsize}}

\noindent%
\begin{tikzpicture}[scale=1]%
\begin{groupplot}[group style={group size=2 by 2, horizontal sep=2cm, vertical sep=1.5cm}]

\nextgroupplot[legend columns=3, legend style={at={(.4,1.05)},anchor=south west}, cycle list name=list_bounds, grid=major, 
   	xlabel={\scriptsize Krylov dimension $m$}, ylabel={\scriptsize Error norm}, ymode = log, xmin = 0, xmax = 90]

\addplot+[mark repeat=10] table [x ={n},y ={opt}] {fig/invsqrt_comparison.dat};\addlegendentry{\scriptsize  Optimal}
\addplot+[mark repeat=10] table [x ={n},y ={lan}] {fig/invsqrt_comparison.dat};\addlegendentry{\scriptsize Lanczos}
\addplot+[] table [x ={n},y ={b1}] {fig/invsqrt_comparison.dat};\addlegendentry{\scriptsize Bound~\eqref{eq:mainresult2}}
\addplot+[] table [x ={n},y ={b2}] {fig/invsqrt_comparison.dat};\addlegendentry{\scriptsize Bound~\eqref{eq:near_spectrum_optimality_invsqrt}/\eqref{eq:near_spectrum_optimality_sqrt}}
\addplot+[] table [x ={n},y ={b3}] {fig/invsqrt_comparison.dat};\addlegendentry{\scriptsize Bound~\eqref{eq:near_fov_optimality}}

\nextgroupplot[cycle list name=list_bounds, grid=major, 
   	xlabel={\scriptsize Krylov dimension $m$}, ylabel={\scriptsize  Error norm}, ymode = log, xmin = 0, xmax=45]

\addplot+[mark repeat=10] table [x ={n},y ={opt}] {fig/invsqrt_comparison2.dat};
\addplot+[mark repeat=10] table [x ={n},y ={lan}] {fig/invsqrt_comparison2.dat};
\addplot+[] table [x ={n},y ={b1}] {fig/invsqrt_comparison2.dat};
\addplot+[] table [x ={n},y ={b2}] {fig/invsqrt_comparison2.dat};
\addplot+[] table [x ={n},y ={b3}] {fig/invsqrt_comparison2.dat};

\nextgroupplot[cycle list name=list_bounds, grid=major, 
   	xlabel={\scriptsize Krylov dimension $m$}, ylabel={\scriptsize  Error norm}, ymode = log, xmin = 0, xmax=90]

\addplot+[mark repeat=10] table [x ={n},y ={opt}] {fig/sqrt_comparison.dat};
\addplot+[mark repeat=10] table [x ={n},y ={lan}] {fig/sqrt_comparison.dat};
\addplot+[] table [x ={n},y ={b1}] {fig/sqrt_comparison.dat};
\addplot+[] table [x ={n},y ={b2}] {fig/sqrt_comparison.dat};
\addplot+[] table [x ={n},y ={b3}] {fig/sqrt_comparison.dat};

\nextgroupplot[cycle list name=list_bounds, grid=major, 
   	xlabel={\scriptsize Krylov dimension $m$}, ylabel={\scriptsize  Error norm}, ymode = log, xmin = 0, xmax=49]

\addplot+[mark repeat=10] table [x ={n},y ={opt}] {fig/sqrt_comparison2.dat};
\addplot+[mark repeat=10] table [x ={n},y ={lan}] {fig/sqrt_comparison2.dat};
\addplot+[] table [x ={n},y ={b1}] {fig/sqrt_comparison2.dat};
\addplot+[] table [x ={n},y ={b2}] {fig/sqrt_comparison2.dat};
\addplot+[] table [x ={n},y ={b3}] {fig/sqrt_comparison2.dat};

\end{groupplot}

\end{tikzpicture}
    
    \vspace{-.25cm}
    \caption{Comparison of the near instance optimality guarantee from~\Cref{thm:near_instance_optimality_stieltjes} to the near spectrum optimality guarantees~\eqref{eq:near_spectrum_optimality_invsqrt} and~\eqref{eq:near_spectrum_optimality_sqrt} from~\cite{AmselChenGreenbaumMuscoMusco2023} as well as the near FOV optimality guarantee~\eqref{eq:near_fov_optimality} for the matrices $A_1$ (left) and $A_2$ (right) defined in the text of \Cref{ex:sharpness}. The vector $\vb$ has normally distributed random entries and $f$ is the inverse square root (top row) or square root (bottom row).}
    \label{fig:comparison_invsqrt_sqrt}
\end{figure}

\begin{example}\label{ex:comparison_amsel_sqrt}
Next, we compare our near optimality guarantees to the near spectrum optimality guarantees~\eqref{eq:near_spectrum_optimality_invsqrt} and~\eqref{eq:near_spectrum_optimality_sqrt} derived in~\cite{AmselChenGreenbaumMuscoMusco2023} for the (inverse) square root as well as to the classical near FOV optimality guarantee~\eqref{eq:near_fov_optimality}. We again use the same experimental setup as in \Cref{ex:sharpness}. As in the examples reported in~\cite{AmselChenGreenbaumMuscoMusco2023}, we use the Remez algorithm for computing the best polynomial approximation on $[\lmin,\lmax]$ or $\spec(A_i), i = 1,2$, respectively. In order to show results up to high precisions, we use variable precision arithmetic via the \texttt{vpa} command from the MATLAB Symbolic Math Toolbox.

The results of this experiment are depicted in~\cref{fig:comparison_invsqrt_sqrt}. As expected, we clearly observe that our new bound from \Cref{thm:near_instance_optimality_stieltjes} much more accurately captures the behavior of the Lanczos approximation. Due to halving the degree of the polynomial approximation in~\eqref{eq:near_spectrum_optimality_invsqrt} and~\eqref{eq:near_spectrum_optimality_sqrt}, it even takes some number of iterations until the improved slope becomes noticeable and these bounds lie below the near FOV optimality guarantee (although this is of course highly dependent on the problem at hand).
\end{example}

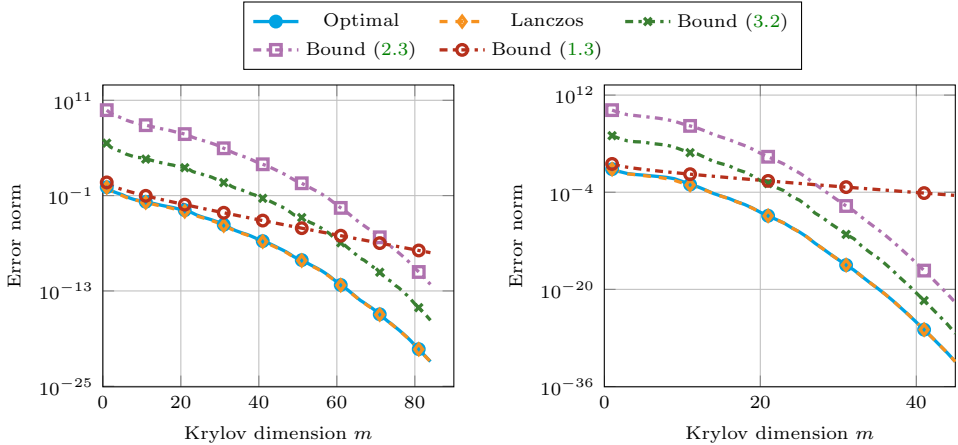
\begin{figure}
    \centering
    \pgfplotsset{height=0.43\linewidth,width=0.48\linewidth,compat=1.10,every axis/.append style={legend style={/tikz/every even column/.append style={column sep=6pt}}}}
\pgfplotsset{every tick label/.append style={font=\scriptsize}}

\noindent%
\begin{tikzpicture}[scale=1]%
\begin{groupplot}[group style={group size=2 by 1, horizontal sep=2cm, vertical sep=2cm}]

\nextgroupplot[legend columns=3, legend style={at={(.4,1.05)},anchor=south west}, cycle list name=list_bounds, grid=major, 
   	xlabel={\scriptsize Krylov dimension $m$}, ylabel={\scriptsize Error norm}, ymode = log, xmin = 0, xmax = 90]

\addplot+[mark repeat=10] table [x ={n},y ={opt}] {fig/log_comparison.dat};\addlegendentry{\scriptsize  Optimal}
\addplot+[mark repeat=10] table [x ={n},y ={lan}] {fig/log_comparison.dat};\addlegendentry{\scriptsize Lanczos}
\addplot+[] table [x ={n},y ={b1}] {fig/log_comparison.dat};\addlegendentry{\scriptsize Bound~\eqref{eq:mainresult2}}
\addplot+[] table [x ={n},y ={b2}] {fig/log_comparison.dat};\addlegendentry{\scriptsize Bound~\eqref{eq:bound_amsel_etal}}
\addplot+[] table [x ={n},y ={b3}] {fig/log_comparison.dat};\addlegendentry{\scriptsize Bound~\eqref{eq:near_fov_optimality}}

\nextgroupplot[cycle list name=list_bounds, grid=major, 
   	xlabel={\scriptsize Krylov dimension $m$}, ylabel={\scriptsize  Error norm}, ymode = log, xmin = 0, xmax=45]

\addplot+[mark repeat=10] table [x ={n},y ={opt}] {fig/log_comparison2.dat};
\addplot+[mark repeat=10] table [x ={n},y ={lan}] {fig/log_comparison2.dat};
\addplot+[] table [x ={n},y ={b1}] {fig/log_comparison2.dat};
\addplot+[] table [x ={n},y ={b2}] {fig/log_comparison2.dat};
\addplot+[] table [x ={n},y ={b3}] {fig/log_comparison2.dat};

\end{groupplot}

\end{tikzpicture}

    \vspace{-.25cm}
    \caption{Comparison of the near instance optimality guarantee from~\Cref{thm:near_instance_optimality_stieltjes} to the near instance optimality guarantee~\eqref{eq:bound_amsel_etal} from~\cite{AmselChenGreenbaumMuscoMusco2023} as well as the near FOV optimality guarantee~\eqref{eq:near_fov_optimality} for the matrices $A_3$ (left) and $A_4$ (right) defined in the text of \Cref{ex:comparison_amsel_log}. The vector $\vb$ has normally distributed random entries and $f$ is the logarithm (or a degree-10 rational approximation of the logarithm for using~\eqref{eq:bound_amsel_etal}).}
    \label{fig:comparison_log}
\end{figure}

\begin{example}\label{ex:comparison_amsel_log}
In our last experiment, we consider the matrix logarithm $\log(A)$, which fits into our framework because $\log(1+z)$ is of the form $zg(z)$ with $g$ Stieltjes. Thus, as long as $\lmin > 1$, we can apply our theory to $\log(I+B)$ with $B = A-I$. We compare with the near instance optimality guarantee~\eqref{eq:bound_amsel_etal} for rational functions and with the near FOV optimality guarantee~\eqref{eq:near_fov_optimality}. As it is done in~\cite{AmselChenGreenbaumMuscoMusco2023}, we construct a degree-10 rational approximation for the logarithm via the BRASIL algorithm~\cite{hofreither2021algorithm} in order to apply the bound~\eqref{eq:bound_amsel_etal}. As the ``Stieltjes formulation'' of the logarithm is only applicable to matrices with $\lmin > 1$, we slightly modify the test matrices from the previous experiments. Specifically, in analogy to $A_1$, we construct $A_3$ with equidistantly spaced diagonal entries ranging from $1.1$ to $110$ and in analogy to $A_2$, the matrix $A_4$ has geometrically spaced eigenvalues in $[1.1, 110]$. In particular, as before, both test matrices have a condition number of $100$. The corresponding results are presented in \Cref{fig:comparison_log}. One shortcoming of our bound~\eqref{eq:mainresult1} that can be observed is that it involves a larger constant now, as it depends on $\kappa(B) \approx 1000$ instead of $\kappa(A) = 100$. The closer the smallest eigenvalue of $A$ is to $1$, the more the bound will deteriorate. The bound~\eqref{eq:bound_amsel_etal} is not affected by this. Instead, the magnitude of its constant depends on the rational function degree that is required for a satisfactory accuracy. Both near optimality guarantees resolve the convergence slope very accurately, while the FOV bound~\eqref{eq:near_fov_optimality} fails to capture the superlinear convergence.
\end{example}

\section{Conclusions}\label{sec:conclusion}
We have proven that the Lanczos method for approximating $f(A)\vb$ is near instance optimal for Stieltjes functions and a related class of functions. Notable functions of interest contained in the two considered classes are the square root, inverse square root and shifted logarithm. We illustrated with examples that the bounds resulting from our near optimality result are much sharper and more predictive of the actual behavior than previously available bounds.

As near instance optimality implies near spectrum optimality, one important consequence of our analysis is that one can analyze the Lanczos method for Stieltjes functions using polynomial approximation on the discrete set of eigenvalues instead of on the field of values.

While our analysis substantially improves over existing results, the constant involved in our bounds is typically still a large overestimate, so that it would be desirable to further reduce it. To foster future improvements in this are, we have illustrated in examples which estimates in our proof are the main cause for the loss of sharpness. 

An obvious direction for future research is the extension of our results to more general $f$, e.g., by using the Cauchy integral formula. This introduces some additional technical difficulties, and experiments reported in~\cite[Appendix~E.3]{AmselChenGreenbaumMuscoMusco2023} suggest that a clean and simple bound with prefactor $C$ independent of $f$ might not be obtainable for general functions (e.g., it is conjectured there that for $A^{-\ell}\vb$, the constant satisfies $C = \Omega(\kappa(A)^{\ell/2})$.

It would also be interesting to generalize our work to the block Lanczos method, in which case the matrix $T_m$ is block tridiagonal and the update considered in~\eqref{eq:low_rank_update} is of a higher rank than two.

\paragraph{Acknowledgment} The author wishes to thank Daniel Kressner for several fruitful discussions on the topic as well as Emil Krieger for helpful comments on an earlier version of this manuscript.

\appendix
\section{Stieltjes and related functions}\label{appendix:stieltjes_functions}
In this section, we review some basic properties of Stieltjes functions and provide some auxiliary results that are required in our derivations. While these results are certainly not new and are well-known to researchers working in the area, it is difficult to find some of the statements in precisely the required form in the literature. We therefore present short proofs for some results to make the treatment as self-contained as possible. Our presentation is inspired by~\cite{AlzerBerg2002,Berg2007,SchillingSongVondracek2012}.

While not directly obvious from the general integral form~\eqref{eq:stieltjes_function}, the class of Stieltjes functions includes many functions of practical interest, including the inverse function $f(z) = \frac{1}{z}$, rational functions in partial fraction form with pairwise distinct, negative poles, $f(z) = \sum_{i=0}^\ell \frac{\sigma_i}{z+t_i}$, where $t_i \geq 0, \mu_i > 0$, inverse fractional powers $f(z) = z^{-\alpha},$ $\alpha \in (0,1)$, and the function $f(z) = \frac{\log(1+z)}{z}$; see~\cite{Berg2007} for proofs that the above are indeed Stieltjes functions as well as for many further examples.

Any Stieltjes function is analytic in the slit plane $\C \setminus (-\infty,0]$ and \emph{completely monotonic}.
\begin{proposition}\label{prop:stieltjes_monotonic}
Let $f$ be a Stieltjes function of the form~\eqref{eq:stieltjes_function}. Then $f$ is completely monotonic, i.e.,
\[
(-1)^k f^{(k)}(z) \geq 0 \text{ for } k \in \mathbb{N}_0 \text{ and } z \in (0,\infty).
\]
\end{proposition}
In particular, \Cref{prop:stieltjes_monotonic} implies that $f$ is nonnegative and monotonically decreasing on $(0,\infty)$.

Several practically important functions are not Stieltjes functions themselves, but of the form $f(z) = zg(z)$, where $g$ is a Stieltjes function. For example $f(z) = \sqrt{z} = zz^{-1/2}$ and $f(z) = \log(1+z) = z\frac{\log(1+z)}{z}$ are of this form. As an easy consequence of \Cref{prop:stieltjes_monotonic}, these functions are nonnegative and monotonically \emph{increasing} on $(0,\infty)$.

\begin{proposition}\label{prop:zf_increasing}
Let $f(z) = zg(z)$, where $g$ is a Stieltjes function. Then $f$ is nonnegative and monotonically increasing on $(0,\infty)$.
\end{proposition}
\begin{proof}
Clearly, $zf(z) \geq 0$ on $(0,\infty)$, as $f(z) \geq 0$ and $z > 0$. To show that $f(z)$ is monotonically increasing, we compute its derivative:
\[
f'(z) = \frac{d}{dz} \left( \int_0^\infty \frac{z}{t + z} \, d\mu(t) \right) = \int_0^\infty \frac{\partial}{\partial z} \left( \frac{z}{t + z} \right) d\mu(t) = \int_0^\infty \frac{t}{(t + z)^2} \, d\mu(t).
\]

Since \( t \geq 0 \), \( z > 0 \), and \( d\mu(t) \geq 0 \), the integrand \( \frac{t}{(t + z)^2} \) is non-negative. Therefore, \( f'(z) \geq 0 \) for all \( z > 0 \).
\end{proof}

A further auxiliary result that we require in the proof of \Cref{lem:f1f2_stieltjes} is given in the following proposition. It gives conditions under which multiplying the integrand in~\eqref{eq:stieltjes_function} by a function $\alpha(t)$ results in a Stieltjes function again.

\begin{proposition}\label{prop:stieltjes_modified_measure}
Let $f$ be a Stieltjes function of the form~\eqref{eq:stieltjes_function} and assume that $\alpha(t)$ is nonnegative on $(0, \infty)$ and goes to zero at least as fast as $\frac{1}{1+t}$, i.e., there exists $c > 0$ such that $\alpha(t) \leq \frac{c}{1+t}$. Then
\[
g(t) = \int_0^\infty \frac{\alpha(t)}{z+t} \dmu
\]
is a Stieltjes function.
\end{proposition}
\begin{proof}
Because $\alpha$ is nonnegative,
\[
\mu_1(t) = \int_0^t \alpha(t) \d\mu(\tau)
\]
is nonnegative and monotonically increasing and because $f$ is a Stieltjes function,
\[
\mu_1(t) \leq c\int_0^t \frac{1}{1+t} \d\mu(\tau) < \infty,
\]
so that it is well-defined. Clearly, 
\[
\int_0^\infty \frac{\alpha(t)}{z+t} \dmu = \int_0^\infty \frac{1}{z+t} \d \mu_1(t),
\]
and we have
\[
\int_0^\infty \frac{1}{1+t} \d\mu_1(t) = \int_0^\infty \frac{\alpha(t)}{1+t} \d\mu(t) \leq \int_0^\infty \frac{c}{(1+t)^2} \dmu < \infty.
\]
Thus, $g$ is a Stieltjes function.
\end{proof}

\bibliography{matrixfunctions}
\bibliographystyle{abbrv}

\end{document}